\documentclass[a4paper,12pt]{article}
\usepackage{amsfonts,eucal}
\usepackage{amssymb,latexsym,amsmath}
\usepackage{amsbsy}

\usepackage{amscd}
\usepackage{parskip}
\usepackage{indentfirst}

\usepackage{longtable}

\usepackage{array, multicol}

\newcommand{\Irr}{{\rm Irr\,}}

\newcommand{\Inf}{{\rm Inf\,}}
\newcommand{\al}{\alpha}

\newcommand{\lam}{\lambda }

\newcommand{\om}{\omega }
\newcommand{\ra}{ \rightarrow }

\def\ii{{if and only if }}

\def\ir{{irreducible }}

\newenvironment{proof}[1][Proof]{\par\noindent{\em #1}.
}{\hfill\framebox(6,6) \par\medskip}

\newtheorem{propo}{Proposition}

\newtheorem{lemma}{Lemma}
\newtheorem{corol}{Corollary}
\newtheorem{theo}{Theorem}

\begin{document}

\title{Representations of dimensions $(p^n\pm 1)/2$ \\ of the symplectic group
of degree  $2n$ \\ over a field of  characteristic $p$\footnote{
Translated from the Russian original in `Vestsi Acad. Navuk BSSR, Ser. Fiz.-Mat. Navuk,
1987, no.6, 9 - 15', (Math. Reviews MR0934193
(89b:20036).)}}
\author{A.E. Zalesskii and I.D.  Suprunenko}

\date{}

\maketitle

\bigskip

\begin{abstract}
The irreducible representations $\phi_n^1$ and $\phi_n^2$ of the
symplectic group $G_n=Sp_{2n}(P)$ over an algebraically closed
field $P$ of characteristic $p>2$
  with  highest weights $\om_{n-1}+\frac{p-3}{2}\om_n$
   and $\frac{p-1}{2}\om_n$, respectively, are investigated. It is proved that the dimension of $\phi_n^i$ ($i=1,2$) is equal to  $(p^n
+(-1)^i )/2$, all weight multiplicities of these representations
are equal to 1, their restrictions to the group $G_k$ naturally
embedded into $G_n$  are completely reducible with  \ir
constituents
 $\phi_k^1$ and $\phi_k^2$, and their restrictions to
$Sp_{2n}(p)$ can be obtained as the result of the reduction modulo
$p$ of  certain complex irreducible representations of the group
$Sp_{2n}(p)$.

These results allow us to obtain the exact list of rational
irreducible representations of simple algebraic groups over fields
of positive characteristics all whose weight subspaces have
dimension 1. This generalizes a result of Seitz.
\end{abstract}

 In this paper we consider the irreducible
representations of the symplectic group $Sp_{2n}(P)$ over an
algebraically closed field $P$ of characteristic $p>2$
  with  highest weights $\om_{n-1}+\frac{p-3}{2}\om_n$
   and $\frac{p-1}{2}\om_n$.  These representations are
   closely linked with each
other and have a number of remarkable properties that motivate an
interest to them. These properties are summarized in the following
theorem.
\medskip

\begin{theo}
\label{th1}  Let $\phi_n^1$ and $\phi_n^2$ be the irreducible
representations  of the group $G_n=Sp_{2n}(P)$ with highest
weights $\om_{n-1}+\frac{p-3}{2}\om_n$
   and $\frac{p-1}{2}\om_n$, respectively  (we omit $\om_{n-1}$ if $n=1$).

\medskip
$(A)$ The dimension of $\phi_n^i$ $(i=1,2)$ is equal to  $(p^n
+(-1)^i )/2$.

\medskip
$(B)$ All weight multiplicities of each $\phi_n^i$ are equal to
$1$.

\medskip
$(C)$ Let the group $G_k$ $(k<n)$ be naturally embedded into $
G_n$. Then the restriction $\phi_n^i |_{G_k}$ is  completely
reducible with \ir constituents
 $\phi_k^1$ and $\phi_k^2$.

\medskip
$(D)$  Let $L_k=G_k\times G_{n-k}$ be the stabilizer of a
non-degenerate subspace of dimension $2k$ in $ G_n$.  Then
$\phi_n^i |_{L_k}$ has exactly two irreducible constituents.

\medskip

$(E)$ The representation  $\phi_n^i|Sp_{2n}(p)$ can be obtained as
the result of the reduction modulo $p$ of a complex irreducible
representation of the group $Sp_{2n}(p)$.
\end{theo}

\medskip
 The statements $(C)$ and $(D)$ will be
refined in the course of the proof. Then this theorem is used to
generalize a recent result of G. Seitz (see below) and to write
down the exact list of rational irreducible representations of
simple algebraic groups over fields of positive characteristics
all whose weight subspaces have dimension 1 (Proposition 2).

{\large Notation and some known facts}.

\smallskip
Recall that $G_n$ is the universal Chevalley group of type $C_n$
over $P$.  Below $X_n$  and $\al_1\ldots \al_n$ are the weight
system and the simple roots of $G_n$, $\al_n$  is a long root, and
$\om_1\ldots \om_n\in X_n$ are the fundamental weights. Weights
from $X_n$ can be written in the form of integral vectors of their
coordinates in the basis $\om_1\ldots \om_n$. If $n > 1$, $\lam =
(a_1\ldots a_n)\in X_n$,  then
 $\lam^0$ is the weight of $X_{n-1}$ obtained from $\lam$
by deletion of $\al_1$. We denote by  $x_\al(t)$ and $X_\al$ the
root elements of $G_n$ and its Lie algebra corresponding to a root
$\al$. We can take $\al_2\ldots \al_n$ as the simple roots of the
subgroup $G_{n-1}\subset G_n$.

\medskip
Set $\om'_n=\om_{n-1}+\frac{p-3}{2}\om_n$ for $n>1$ and $\om'_n=
\frac{p-3}{2}\om_n$ for $n=1$, $\om''_n= \frac{p-1}{2}\om_n$,
$H_n=Sp_{2n}(p)$. For $k<n$,  we assume that the group $H_k$ is
 naturally embedded into $H_n$. If  $F_k$ is the stabilizer in $H_n$ of a
non-degenerate subspace of dimension $2k$, then $F_k\cong
H_k\times H_{n-k}$.

\medskip
(I) We describe a construction of the complex irreducible
representations  $\theta_r^i$ ($i=1,2$) of $H_r$ of degrees $(p^r
+(-1)^i)/2$. Let $A_r\subset GL(p^r, {\mathbb C})$  be an
irreducible group containing the subgroup $Z$ of all scalar
matrices and such that $A_r/Z$ is an abelian group of the exponent
$p$. Let $N_r$ be the normalizer of $A_r$ in $GL(p^r,{\mathbb
C})$. Then $N_r/A_r\cong H_r$ (\cite[\S 20]{Su}).   For $p>2$, the
group  $N_r$ splits: $N_r=A_rH_r$. In order to show this, we fix
an involution $i\in N_r$ whose image is a central element in
$N_r/A_r$, and consider the group $D=C_{N_r}(i)$. It is  not
difficult to see that $D\cap A_r=Z$ and $D/Z\cong H_r$. The
splitting of $D$, i. e. the equality $D=ZH_r$, follows from the
fact that $H_r$ coincides with its  derived subgroup and has no
non-splitting central
 extension (\cite[Theorem 10 and Corollary 2]{St}), except for the case
$r=1$, $p=3$, which can be easily handled directly. It is well
known that the representation of $H_r$ obtained in this way has
two irreducible constituents of dimensions $(p^r -1)/2$ and $(p^r
+1)/2$ \cite{W,Is,Se5,Za}. This determines the irreducible
representations $\theta_r^1$ and $\theta_r^2$ of these dimensions
of the group $H_r$ (in general, not uniquely).

 The well known
equality $A_r=A_{r-1}\otimes A_1$ (the tensor product) determines
the embedding $N_{r-1}\ra  N_{r-1}\otimes E\subset  N_r$, what
allows to choose $\theta_n^i$ $(i=1,2,$ $n=1,2,\ldots )$ so that
the restriction $\theta_n^i|_{H_{n-1}} $ consists   of the
constituents $\theta_{n-1}^1$ and $\theta_{n-1}^2$ (disregarding
 multiplicities). In \cite[Theorem 2]{Za} the following
formulas are proved:

\begin{equation}\label{eq1}
\theta_n^1|_{F_k} =\theta_k^1\otimes\theta_{n-k}^2\oplus
\theta_k^2\otimes\theta_{n-k}^1, ~~~~~\theta_n^2|_{F_k}
=\theta_k^1\otimes\theta_{n-k}^1\oplus
\theta_k^2\otimes\theta_{n-k}^2 .\end{equation}

\medskip
(II) Let $\phi$ be an irreducible complex representation of $ H_1$
such that $\dim \phi =(p-1)/2$ or $(p+1)2$. Then $\phi$ remains
irreducible under reduction modulo $p$.

\medskip
This follows from the comparison of ordinary and Brauer characters
of the corresponding degrees, which can be easily done using the
character table from  \cite{Sp}.

\medskip
Ward \cite{W} has shown that the representations $\theta_n^i $
remain irreducible under reduction modulo $r$ for an arbitrary odd
prime  $r\neq p$, see also \cite{Se5}. The following lemma extends
this result to the case $r=p$.

\medskip
\begin{lemma}\label{1}
 Let  $\overline{\theta_n^i} $ be the reduction
  modulo $p$ of the
representation  $\theta_n^i $
 $ (i=1 ,2)$. Then $ \overline{\theta_n^i}$ is irreducible
 and the restriction $\overline{\theta_n^i}|_{F_k}$ is
completely reducible.
\end{lemma}

Proof. For $n=1$ see (II). Then we use induction on $n$. Suppose
that $\overline{\theta_n^i}$ is reducible; let $\rho$ be  one of
its irreducible constituents. The induction assumption and
Formulae (1) imply
$$\rho|_{F_k}=\overline{\theta_k^l}\otimes
\overline{\theta_{n-k}^m}, ~~~~ (1 \leq l,m\leq 2).$$ But this is
false since otherwise the non-central elements from $H_r$ lying in
the center of $F_k$ would have scalar images in $\rho$. The second
statement of the theorem follows from Formulae (l) and the fact
that the center of $F_k$ consists of semisimple elements and
differs from the center of $H_n$. This completes the proof.

There is an element $s$ of order $p^n+1$ generating a subgroup
which is irreducible in the natural representation  of $H_n$ (see
\cite[Theorem 2]{Hu}).

\medskip
{\bf Lemma 2} {\it All eigenvalues of $\theta_n^i(s)$ $(i=1 ,2)$
have multiplicity $1$.}

\medskip
Proof. The representation $\theta_n^i$ is an irreducible
constituent of the representation $\tau$ of degree $p^n$ of the
group $H_n$ described in (I). In fact, it is proved in
\cite[pp.716-717, 705-706 and formula (5)]{Sh} that all
eigenvalues of $\tau(s)$ have multiplicity 1. Of course, this is
also true for the composition factors of $\tau$.


\begin{propo}\label{pr5}
  {\it Let $\Lambda$ be an infinitesimally
irreducible representation of $G_n$. Suppose that  the irreducible
representations with  highest weights $\om_{n-1}',\om''_{n-1}$ are
the only
 composition factors of the restriction
$\Lambda|_{G_{n-1}}$. Then either  the highest weight of $\Lambda$
is  $\om_{n-1}'$ or $\om_n'$, or  $n=2$, $p=3$,  and $\dim \Lambda
=1$.}
\end{propo}

\begin{proof} Let $\Lambda$ be realized in a space $V$,  and let
$\lam =(a_1\ldots a_n)$ be the highest weight of $\Lambda$. Let
$v\in V$ be a non-zero vector of  weight $\lam$. Set
$v_j=X^j_{-\al_1} v$ $(0\leq j\leq a_1)$.  Then $X_{\alpha_1}v_{j+1}=(j+1)(a_1-j)v_j$. Since $a_j<p$, this implies that
$v_1, \ldots v_{a_1}\neq 0$. As $x_{\al_k} (t)$ and $X_{-\al_1}$ commute for
$k>1$ and  $t\in P$, we have  $x_{\al_k}
(t)v_j=v_j$. So the module $G _{n-1}v_j$ has a composition
factor of highest weight $(\lam-j\al)^0$. Since
 $\al_1=2\om_1-\om_2$,  we have $(\lam-j\al_1) ^0=( a_2+j, a_3\ldots a_n)$.
According to the assumption,  this vector is equal to $(0\ldots 0
, 1,( p -3 )/2  )$ or $(0\ldots 0 , ( p -1 )/2  )$ (for $n=2$ the
first vector is equal to $(p-3)/2))$. For $n>2$ and $a_1>0$,  we obtain
a contradiction by taking in turn $j=0,1$. This implies the
proposition for $n>2$. Let $n=2$. If $a_1>0$, then we obtain the
unique possibility $(1,(p-3)/2)$ by considering the cases
$j=0,1,2$ for $a_1\geq 2$, and $j=0,1$ for $a_1=1 $. Let $a_1=0$.
Then either $a_2=(p-1)/2 $ (as required), or $a_2=(p-3)/2.$ In the
latter  case we obtain $\dim \Lambda=1$ for $p=3$. Let $p>3$. Set
$u=X_{-\al_1} X_{-\al_2}v$. It is not difficult to verify that $x
_{\al_2}(t)u=u$, $X_{\al_2}X_{\al_1}u=(p-3)v$, hence $u\neq 0$.
So the module $G_{n-1} u$ has a composition factor of highest
weight $(\lam-\al_1-\al_2)^0=((p-5)/2)$, which contradicts the
assumption. This completes the  proof.
\end{proof}

{\bf Proof of the Theorem}. It is well known that every
irreducible $P$-representation of $H_r$ can be extended to an
infinitesimally irreducible representation of $G_r$. We denote by
$\Theta^n_i$ such extension  of $\overline{\theta_n^i}$ to
$G_n$ ($i=1,2$).

\medskip
We show that $\Theta^n_i$ satisfies the assumptions  $(A)$ - $(E)$
of the theorem. $(A)$ and $(E)$ follow from the construction of
$\theta^n_i$. By Lemma 2, all eigenvalues of the matrix
$\Theta^n_i(s)=\overline{\theta_n^i}(s)$ have multiplicity 1 for
some semisimple element $s\in H_n$. This implies $(B)$ since $ s$
is contained in a maximal torus of $G$ (\cite[Corollaries 19 and
21B)]{Hu}). Include $F_k$ into $L_k$. By Lemma 1, Formulae (1)
hold if one replaces $\theta_n^i$ by $\overline{\theta_n^i}$, and
hence the restriction $\Theta_n^i |_{L_k}$ has two composition
factors. It  is completely reducible since the centers of $F_k$
and $L_k$ coincide. This implies $(D)$.

\medskip
Let $\lam_i=(a_1\ldots a_n )$ be  the highest weight of
$\Theta_n^i$, and $v\neq 0$ a vector of weight $\lam_i$ in the
space affording
 $\Theta_n^i$. If $ n=1$, we have $\dim \Theta_n^i=(p\pm 1)/2$,
so $\lam_1=\om_1'$ and $\lam_2=\om_1''$. To prove this for $n>1$,
we use induction.  We  show that the composition factors of the
restriction $\Theta_n^i |_{ G_{n-1}}$ $(n>1)$ are infinitesimally
irreducible. By $(D)$, it is sufficient   to find two such
non-isomorphic factors. Obviously, one of them is of highest
weight $(a_2\ldots a_n)$.  So $(a_2\ldots a_n)=\om_{n-1}'$ or
$\om_{n-1}''$  by Lemma 1  and Formulae (1). Let $j=\,$min$\, \{k:
~1\leq k\leq n,~ a_k\neq 0\}$. Arguing as in the proof of
Proposition 1, we observe that $\Theta_n^i|_{G_{n-1}}$ has an
infinitesimally irreducible composition factor of highest weight
$(\lam_i-\al_1-\cdots -\al_j)^0\neq (a_2\ldots a_n)$ which belongs
to the module $G_{n- 1}
 u$, where $u=X_{-\al_1}\cdots X_{-\al_j}v$. (Here we take into account the
commutation relations in the Lie algebra of $G_n$  and
\cite[Proposition 5.4]{Sp}.) Now, due to the induction hypothesis,
all composition factors of the restriction $\Theta_n^i |_{ G_{k}}$
are infinitesimally irreducible for $k<n$. This is also true for
$L_k$, so Formulae (1) hold if we replace $\theta_n^i$ by
$\Theta_n^i$. This implies $(C)$.

Thus, the representations $\Theta_n^i$ satisfy the assumptions of
Proposition 1  and $\dim \Theta_n^i>1$. Hence $\lam_i\in
\{\om_n',\om_n''\}$. Note that the roots $\al_2,\ldots ,\al_n,
2\al_1+\cdots +2\al_{n-1}+\al_n$  form a basis of the root system
of $L_{n-1}$. Therefore, the module $L_{n-1}v$ contains a
composition factor $M$ of highest weight $\om=(\lam_i)^0\times
(\sum_{j=1}^n a_j)$. If $\lam_i= \om_n'$, then $\om =
\om'_{n-1}\times \frac{p-1}{2}$, and if $\lam_i=\om_n''$, then
$\om = \om''_{n-1}\times \frac{p-1}{2}$. This implies that
$M=\Theta^1_{n-1}\otimes \Theta^2_1$ in the first case  and
$M=\Theta^2_{n-1}\otimes \Theta^2_1$ in the second one (due to the
induction hypothesis). Now $(D)$ yields  that $\lam_1= \om_n'$ and
$\lam_2=\om_n''$, which proves  the theorem.


\begin{corol}\label{c1}
Let $\psi_1,\psi_2$  be the irreducible representations of
$Sp_{2n}({\mathbb C})$ with  highest weights $\om_n',\om_n''$,
respectively. Then the number of weights of $\psi_i$ is equal to
$\frac{p^n+(-1)^i}{2}$ $(i=1,2)$.
\end{corol}

\begin{proof} By Premet's theorem \cite{P}, the numbers of  weights
of $\psi_i$ and $\phi_i$ are equal for $p>2$. It remains to use
assertions  $(A)$ and $(B)$ of the theorem.
\end{proof}

Let $G$ be a simply connected  simple algebraic group over $P$; below $p$
may  be equal to  2. Let $B=(\al_1,\ldots ,\al_n)$ be the set of
 simple  roots of $G$,  and let $\om_1\ldots \om_n$ be the fundamental
weights labelled as in \cite{Bo}.  Below $\Irr G$ and $\Inf G$ are
the sets of irreducible rational and infinitesimally irreducible
representations of $G$, respectively, $X(G)$, resp., $X^+(G)$ is
the set of its weights, resp., dominant weights; $X(\phi)$ and
$\om(\phi)$  are the weight system and the highest weight of a
representation $\phi\in\Irr G$. If $\om\in X^+(G)$, then
$\phi(\om)$ is the irreducible representation of $G$ with highest
weight $\om$.  For $p=2$, $G=C_n(P)$ or   $F_4(P)$, and for $p=3$,
$G=G_2(P)$, we set $\Inf_1 G$ (respectively, $\Inf_2 (G)$)) to be  equal to
$\{\phi\in\Inf  G\, |\,\langle \om(\phi),\beta\rangle=0\}$ for all
long (respectively, short) roots $\beta\in B\}$,
$\Inf{}'\,G=\Inf_1G\cup\Inf_2\, G$. Here $\langle
\om(\phi),\beta\rangle$ is defined as in \cite[\S 3]{St}.  Let
$I(G)$ be the set of representations $\phi\in \Irr G$ all whose
weight multiplicities are equal to 1, and let $I_0(G)=I(G)\cap\Inf
G$, \ $I'_0(G)=I(G)\cap \Inf'G$ in those cases where the set $\Inf'
G$ is defined.

\medskip
 In the course of the investigation of
irreducible embeddings of simple algebraic groups Seitz \cite{Se}
has singled out a certain set $M$ of \ir representations with the
property $I_0(G)\subset M \subset\Inf G$ and  $I_0'(G)\subset M \subset
\Inf' G$ if the sets $I'_0(G)$ and $\Inf' G$ are defined. This was
sufficient for his purpose, so he did not consider the problem of
 determining  $I(G)$. This problem seems to us to be rather
important, so we  continue the analysis of Seitz' list to
determine $I(G)$. Here \cite[Theorem 6.1]{Se} and statement $(B)$
of our main theorem are essentially used.

According to Steinberg \cite[Theorem 11.1]{St},  in the cases
where  the sets $\Inf_iG$ $(i=1,2)$ are defined, each
representation $\phi\in\Inf G$ is of shape $ \phi_1\otimes \phi_2$
where $\phi_i\in\Inf_i G$ (see also \cite[Corollary of Theorem
41]{Ste});   here it is obvious that  if $\phi\in I(G)$, then
$\phi_1,\phi_2\in I'_0(G)$. Taking into account the natural
isomorphism of the weight systems of $ G$ and the corresponding
Lie algebra over the complex numbers, we observe that
$\phi(\om)\in I_0(G)$  if $\om$ is a miniscule weight (\cite[Ch.
VIII, \S 2.3]{Bo}). In this case the Weyl group is transitive on
$X(\phi(\om))$, and hence all weights have multiplicity $1$.

\begin{corol}\label{cc2}
Let $G=C_n(K)$ and  $n>1$. A non-trvial representation $\phi\in\Inf G$
with  highest weight $\om$ belongs to $I(G)$ \ii either
$\om=\om_n$, $n=2,3$, or $p=2$ and $\omega=\omega_1$ or
$\omega_n$, or $p>2$
 and $\om\in\{\om_1,\om_n', \om_n''\}$.
\end{corol}

\begin{proof}  First let $p>2$. By \cite[Theorem 6.1]{Se}, if $
\phi\in I_0(G_n)$,  then either
$\om\in\{\om_1,\om_n,\om_n',\om_n''\}$, or $p=3$ and  $\om=\om_i$
$ (i=1,\ldots,  n)$. It  is clear that $\phi(\om_1)\in I(G)$ (for
$p=2$ also) since  $\om_1$ is a miniscule weight. By the assertion S
$(B)$ of the theorem, $\phi(\om_n'),\phi(\om_n'')\in I(G_n)$. It
follows from the description of a basis of the Weyl module $V_r$
with highest weight $\om_r$ for the group $G$, see \cite[Theorem
1, p. 1324]{PS}) that for $r>1$,  the multiplicity of the weight
$\om_{r-2}$ in $V_r$ is equal to $n-r+1$ and the multiplicity of
the zero weight in $V_4$ for $n=4$ equals  $2$. By \cite[Theorem 2(i)]{PS},
 every composition factor of $V_r$ occurs with multiplicity 1.
 Therefore,
the multiplicity of the weight $\om_{r-2}$ in $M_r:=\phi(\om_r)$
is at least $n-r$ since the weight $\om_{r-2}$  can only occur in
 the composition factors of $V_r$ with  highest weights $\om_r$ and
$\om_{r-2}$. Therefore,  $\phi(\om_r )\not\in I(G_ n)$ for $1<r<n-1$.

 Observe that $\om_n'=\om_{n-1}$ and $\om_n''=\om_n$ for $p=3$.

The case where $p>3$ and  $\om=\om_n$.  For $n=4$,  the module $V
_4$ is irreducible (see Example 5 at the end of \cite{PS}), and as
we have noted before,  the multiplicity of the weight 0 in $M_4$
is equal to $2$. So $\phi(\om_4  )\notin I(G _4)$. Let $v\neq 0$
be a vector of  weight $\om_n$ in the $G_n$-module $M_n$ $(n>4)$.
Then the $G_4$-module $G_4v$ has a composition factor $M_4$.  This
implies that $M_n$ has a weight subspace  of dimension at least $
2$, that is, $\phi(\om_n)\notin G_n$.

 For $n=2,3$,  one can easily  check that
all weight subspaces of $V_n$, and hence of  $M_n$, have dimension
$1$.

 Let $p=2$.    If $\phi\in I_0'(G )$, then
   $\om\in\{\om_1,\om_n\}$ by \cite[Theorem 6.1]{Se}.   For $p=2$,
   we have  $B_n ( P ) = C_ n ( P )$
and $\phi(\om_n)\in  I(  B_ n(P  ) )$ since $\om_n$ is a miniscule
weight of $B_n(P) $. Therefore, $\phi(\om_n)\in  I'_0( C_ n(P  )
)$.

Due  to  the  remark prior to  Corollary 2,  it remains to
consider the case where $\phi=\phi(\om_1)\otimes \phi(\om_n)$.
Note that $\om_n-\al_n\in X(\phi(\om_n))$, and $\pm
\frac{\al_n}{2}\in X(\phi(\om_1))$. So
$\om_n-\frac{\al_n}{2}=(\om_n-\al_n)+\frac{\al_n}{2}\in X(\phi)$
is of multiplicity  $\geq2$, that is, $\phi\notin I(G)$.  This
completes the proof.
\end{proof}

It is well  known that $I ( G ) = $\Irr$\, G$ if $G = A_1( P )$.

\begin{propo}\label{PR2}
  Let $G\neq A_1(P)$    be a simply connected  simple algebraic
group over an  algebraically closed field $ P$   of characteristic
$ p>0$. Define the set of weights $\Omega=\Omega(G)$ as follows:

{\small
$$\Omega(A_n(P) ) =\{\om_i, a\om_1,b\om_n, c\om_j+
(p-1-c)\om_{j+1}~(1\leq i\leq n,~  1\leq j< n,~ 0\leq
a,b,c<p\};$$

$$\Omega(C_n(P) ) =\{\om_n~{\rm  for}~ n=2,~~~~~
\om_1,\om_{n-1}+\frac{p-3}{2}\om_n, \frac{p-1}{2}\om_n\}  ~{\rm  for}~ p>2$$
$$~{\rm
and}~~~~\{\om_1, \om_n\}~{\rm  for}~p=2;$$}

 $$\Omega(B_n(P) ) =\{\om_1, \om_n\}~~{\rm for}~~ n> 3,p>2; ~~~~~\Omega(D_n(P) )
=\{\om_1, \om_{n-1}, \om_n\};$$
$$\Omega(E_6(P) ) =\{\om_1,
\om_6\};~~~~~\Omega(E_7(P) ) =\{ \om_7\};$$  $$\Omega(F_4(P) )
=\{\om_4\}~{\rm  for}~p=3~{\rm  and}~ \emptyset~ {\rm
for}~p\neq3;$$
$$\Omega(G_2(P)
) =\{\om_1\}~{\rm  for}~p\neq 3~{\rm  and}~ 
\{\om_1,\om_2\}~{\rm  for}~p=3.$$ \smallskip Let $\phi$ be an irreducible
rational representation of $G$ with highest weight
$\om=\sum_{i=0}^k p^i\lam_i$ where $\lam_i$ are the highest
weights of representations from {\rm Inf}$(G)$. The multiplicities
of  all weights of   $\phi$ are equal to $1$ if
and only if $\lam_l=0$ or $\lam_l\in\Omega (G)$  for $0\leq l\leq
k$, and $\lam_{l+1}\neq \om_1$  provided $G=C_n(P)$, $p=2$, $
\lam_l=\om_n$, or $G=G_2(P)$,  $p=2$, $\lam_l =\om_1$,  or
$G=G_2(P)$, $p=3$,  $\lam_l =\om_2$.
\end{propo}

 \begin{proof}  First, let  $ \phi\in\Inf G$ and $\om\neq 0$.
The case where $G = C_n(P)$  has been   considered in Corollary 2. Let $G
\neq C_n(P)$ and $G \neq  G_2(P)$ if  $p=3$. Then $\om\in\Omega $
for $\phi\in I_0(G)$ by \cite[Theorem 6.l]{Se} and the observation
prior to Corollary 2, and it suffices  to consider the cases where
$\om$ is not a miniscule weight. It is well known that $\phi \in
I_0( G)$ if $G = B_ n ( P)$ or $G_2(P)$ and $\om=\om_1$ (this is
also true for $p=3$.) For $p=3$ and $G=F_4(P)$,  Wong
\cite[p.4]{W} has shown that $\phi(\om_n)\in I_0 ( G )$. It is
proved in \cite{S1} that $\phi\in I_0(G)$  if $G=A_n(P)$ and
$\om=c\om_i+(p-1-c)\om_{i+1}$ $(1\leq i\leq n)$ (\cite[Theorem
3]{S1}) or $\om=a\om_1$ or $b\om_n$ (\cite[Remark 1]{S1}). In the
latter case   this is also mentioned  in \cite[1.14]{Se}.

\medskip
 Now let $G=G_2(P)$ for $p=3$. By  \cite[Theorem 6.1]{Se},
 $\om=\{\om_1,\om_2\}$ for $\phi\in I_0' (G)$. For $p=3$, the representation $\phi
(\om_2)$ belongs to  $I_0'(G)$ since it can be obtained from
$\phi(\om_1)\in I_0'(G)$ by twisting it with  an automorphism of
$G$ (see \cite[Corollary 11.2]{St}). Due to the remarks prior to
Corollary 2, it remains to consider the representation
$\phi=\phi(\om_1)\otimes \phi(\om_2)$. Since $\om_1-\al_1,
\om_1-\al_1-\al_2\in X (\phi(\om_1))$  and $\om_2,\om_2-\al_2\in X
(\phi(\om_2))$,  the weight $\om_1+\om_2-\al_1-\al_2\in X (\phi)$
has  multiplicity at least 2, i.e.  $\phi\notin \Irr G$.

\medskip
 Now  let  $\phi\in $Irr$\,G$. By  \cite[Theorem 41]{Ste},
$\phi = \otimes_{l=0}^{k}Fr^l\circ \phi(\lam_l)$ where $Fr$ is the
Frobenius morphism associated with raising elements of $P$ to the $p$th power. If $ \phi\in I (G)$,  then it is obvious that
$\phi(\lam_l)\in I_0(G)$;  it follows from the above that every
weight $\lam_l\in \Omega\cup \{0\}$.

\medskip
Let $ \rho\in I_0(G)$,  $\psi\in I(G)$. It is clear that the
representation $\rho \otimes Fr\circ \psi\notin I(G)$  if and only
if

\begin{equation}\label{eq2}
\mu_1-\mu_2=p(\mu_1'-\mu_2')\neq 0\end{equation}

\noindent for some weights $\mu_i\in X (\rho)$, $\mu_i'\in
X(\psi), $ $i=l,2$. Set $\mu_1-\mu_2=\mu$ and S $\om' =\om(\rho)$.
Acting by the Weyl group, we make  $\mu$ a dominant weight. Observe
that $\mu$  and $\mu_1'-\mu_2'$ are radical weights. Let $\al$ be
the maximal root of $G$, $a(\lam)=\langle \lam,\al\rangle $ for
$\lam\in X(G)$,  and $a=a(\rho)=
  a(\om')$.    As it is noted in \cite[Lemma
11]{S2}, $a(\lam)\leq a$ for $\lam\in X (\rho)$. In order to find
$a$, one may apply
 formulae from \cite[Introduction]{S2}
   or   proceed by     direct calculations
using the root tables from \cite{Bo}.  Note that  $a (\lam )$ is
an integer valued linear function  on $X(G)$ (see \cite[\S
3]{Ste}); $a(\lam)\neq 0$ for $\lam\in X^+(G)\setminus \{0\}$. Let
$\rho^*$ be the representation dual to $\rho$.  By \cite[Lemma
73]{Ste}, $\om(\rho^*)=-w_0\om'$ where $w_0$ is the element of the
Weyl group sending all positive roots to the negative ones.
 It is clear
that $-w_0\al=\al$. So $ a (\rho^* )=a$, $ a (\lam )\geq -a$   for
$\lam\in X(\rho)$ since $-\lam\in X(\rho^*)$, and $a(\mu)\leq 2a.$

\medskip
We shall consider all representations  $\rho$ with $\om'\in\Omega$
 and  find
out when (2) holds. If $G=A_n(P)$ or $C_n(P)$, then
$a=\sum_{i=1}a_ i$ for $\om'=( a_1, \ldots,     a_n   )$.

\medskip
 If $G=A_n(P)$,
then $a(\rho)\leq p-1$ for $\om'\in\Omega$.  In this  situation
(2) would imply that $ a(\mu)=p$, $\mu'_1-\mu_2'=\om_i$  which is
false since $\om_i$ is not a radical weight.

\medskip
Let $G=C_n(P)$ and   $n>2$. For $p>2$,  (2) cannot hold since
$a(\rho)\leq (p-1)/2$ for $\om'\in\Omega$ and  $a(\mu)\leq p-1 $. If
$p=2$, then the weight $\mu$ takes  each value $2\om_i$, $i\leq n$
for $\om' =\om_n$ and $\mu=2\om_1$ or $\om_2$ for $\om'
=\om_1$. (This can be easily verified taking into account that
$C_n(P)\cong B_n(P)$ and that $\om_n$ a miniscule weight of $B_n(
P)$.) Therefore,  (2)  is equivalent to the fact that  $\om'
=\om_n$, $\psi=\phi(\om_1) \otimes Fr\circ \psi'$ where $\psi'\in
\Irr G$.

\medskip
 Let
$ G=G_2(P)$. Then one can directly verify   that
$\mu\in\{\om_1,\om_2,2\om_1\}$ for $\om' =\om_1$  and
$\mu\in\{\om_2,2\om_2,3\om_1\}$ for $\om' =\om_2$, $p=3$; all
these options  are realized.  It is clear that (2) holds in the following
cases only: $p =2$,  $~\om' =\om_1$, $~\psi=\rho\otimes Fr\circ
\psi'$ and $p=3$, $~\om' =\om_2$, $~\psi=\phi(\om_1) \otimes
Fr\circ \psi'$,  $~\psi'\in \Irr G$.

\medskip
Let $G\neq C_n(P)$, $A_n(P)$, or $G_2(P)$. In this situation
$a(\rho )= 1$ for $\om'\in\Omega$. So $a(\mu)\leq 2$.  Since for
$G=D_n(P)$, $E_6(P)$ or $E_7(P)$, the weight $\lam\in X^+(G)$ is
not radical  if $ a(\lam)=1$, and $p>2$ for $G=B_n(P)$ or
$F_4(P)$, the condition (2) does not hold.

Taking into account  that representations $\psi$ and $Fr^j\circ
\psi $ from $\Irr G$ simultaneously belong or do not belong to
$I(G)$, we complete  the proof by induction   on $k$.
\end{proof}



\newpage


\bigskip
Institute of Mathematics of Acad. Sci. of BSSR~~~~~~~~~~~~~Received 16.4.1986


\bigskip
{\bf Remarks on translation.} (1) In the proof of Proposition 1 one could apply \cite[2.5]{Se} in order to conclude that the vectors $v_j$ are nonzero. (2) Now one can cite the article: A.E. Zalesskii, I.D. Suprunenko, Reduced symmetric powers of natural realizations of the groups $SL_m(P)$ and $Sp_m(P)$ and their restrictions on subgroups, Siberian Math. J., 31:4 (1990), $555-566,$ Proposition 1.4, instead of the preprint \cite{S1}.

\end{document}